\documentclass{amsart}
\usepackage{amsfonts}
 \usepackage{amssymb}
 \usepackage{verbatim}
 \title[Generalized Wolff's Ideal  Theorem]{A Generalized Wolff's Ideal  Theorem  on Certain Subalgebras of
$H^{\infty}(\D)$}
\author{Debendra P.  Banjade, Caleb D. Holloway, and Tavan T. Trent}

\address{Department of Mathematics and Statistics\\
         Coastal Carolina University \\
         P.O. Box 261954\\
         Conway, SC 29528-6054 \\
         (843) 349-6569}
\email{dpbandjade@coastal.edu}

\address{Department of Mathematical Sciences\\
	309 SCEN\\
	1 University of Arkansas\\
	Fayetteville, AR 72701\\
	(970) 943-2015}
\email{cdhollow@uark.edu}

\address{Department of Mathematics\\
         The University of Alabama\\
         Box 870350\\
         Tuscaloosa, AL  35487-0350\\
         (205) 758-4275}
\email{ttrent@as.ua.edu}

 \newtheorem{thm}{Theorem}[section]
\newtheorem*{thm*}{Theorem}  \newtheorem{cor}{Corollary}[section]
\newtheorem{lemma}{Lemma}[section]
\newtheorem{rem}{Remark}  
 
\newtheorem*{wolff}{Theorem A (Wolff)}

\newtheorem*{tr}{ Theorem B (Treil)}
\newcommand{\ran}{\operatorname{ran}}  \newcommand{\all}{\;
\forall \;} \newcommand{\exist}{\; \exists \;} \newcommand{\C}{\mathbb{C}}
  
\newcommand{\I}{\mathcal{I}} \newcommand{\D}{\mathbb{D}} \newcommand{\Z}{\mathbb{Z}}
  \newcommand{\rad}{\operatorname{Rad}}

\allowdisplaybreaks[1]

\begin{document}

\maketitle

\begin{abstract} We prove the generalized Wolff's Ideal Theorem  on certain uniformly closed subalgebras of
$H^{\infty}(\D)$ on which the Corona Theorem is already known to hold. \end{abstract}

\section{Introduction}

Carleson's celebrated proof of the Corona Theorem \cite{carleson}, which gives necessary and sufficient conditions for unit membership in the ideal of $ H^{\infty}(\D) $ generated by a given set of functions, opened the door for several new questions. The first we will consider, which we call a ``generalized ideal problem,'' asks whether we can find  weaker conditions under which a given function $ h $ is included in the ideal. If not, can we at least find some $ p > 1 $ so that $ h^p $ belongs to the ideal? Also, are there other algebras for which a result similar to Carleson's holds?

The first two questions were proposed and answered (at least in part) by  Wolff
\cite{wolff} in a result we refer to as \lq\lq Wolff's Theorem.'' The third has been a topic of research
over the years, with varying results. (For examples of algebras on which a corona theorem holds, see
Tolokonnikov \cite{tolokonnikov} and Nikolski \cite{nikolski}(Appendix 3, P. 288) as well as Costea-Sawyer-Wick
\cite{druryarveson}; for some negative examples, see Scheinberg \cite{scheinberg} and Trent
\cite{anote}).
 
Carleson's Corona Theorem states that the ideal $ \I $ generated by a finite set of functions $
\lbrace f_i \rbrace_{i=1}^{n} \subset H^{\infty}(\D) $ is the entire space $ H^{\infty}(\D) $
provided that there exists $ \delta > 0 $ such that 
\begin{equation}  \label{eq:corona}
\left( \sum_{i=1}^{n} |f_i(z)|^2\right)^{\frac{1}{2}} \geq \delta  \text{ for all } \; \;  z \in \D \text{.} \end{equation} This result can be extended to hold for infinitely many
functions $ \lbrace f_i \rbrace_{i=1}^{\infty} $ (see \cite{rosenblum}, \cite{operatortheory}).

Under what conditions, then, could we expect a given function $ h \in
H^{\infty}(\D) $ to be found in $ \I $? One might suppose, based on Carleson's result, that  a
sufficient condition would be  
\begin{equation}\label{eq:wolff} 
\left( \sum_{i=1}^{n} |f_i(z)|^2\right)^{\frac{1}{2}} \geq |h(z)| \text{ for all } \; \;\;  z \in \D \text{.} 
\end{equation} 
Obviously, the condition  \eqref{eq:wolff} is necessary, but it is not sufficient. (See  Rao's example in Garnett  \cite{garnett}, P. 369, Ex-3.) However, Wolff  (see Garnett \cite{garnett}, P. 329, Theorem 2.3) proved that, given \eqref{eq:wolff}, $ h^3 \in \I .$

\begin{wolff}
If
\begin{align*}
f_j \in H^{\infty}(\mathbb{D}), j=1, 2, ..., n,\;   h \in H^{\infty}(\mathbb{D}) \quad \text{and}\notag\\
\left( \sum_{j=1}^n \, |f_j(z)|^2\right)^{\frac 1{2}} \geq |h(z)| \; \text{ for all } \; \;\;  z \in \mathbb{D},
\end{align*}
then
\[
h^3 \in \mathcal{I} ( \{ f_j\}_{j=1}^n),
\]

\medskip \noindent
the ideal generated by $\{f_j\}_{j=1}^n$ in $H^{\infty}(\mathbb{D}).$
\end{wolff}

Further research has revealed conditions under which $ h $ itself may be contained within the ideal. For $ f_j \in H^{\infty}(\D), j=1, 2, \dots $, let $ F(z) = (f_1(z) , f_2(z) , \dots ) $, and let $ F(z)^{\ast} $ denote the adjoint of $ F(z) $ for $ z $ fixed. If, for $ h \in \D $ and $ p > 1 $, we assume
\begin{equation*}
1 \geq [F(z)F(z)^{\ast}]^p \geq | h(z) | \text{ for all } z \in \D \text{,}
\end{equation*}
then we obtain $ h \in \mathcal{I} ( \{ f_j\}_{j=1}^n) $. This was shown by Cegrell \cite{cegrell1} for finitely many functions $ f_j $, and extended to infinitely many functions by Trent \cite{estimate}. It was known that the result fails if $ p < 1 $ for some years before Treil proved that it also fails if $ p = 1 $ \cite{treil2}.

Can this estimate be improved upon? More precisely, can we find a non-decreasing function $ \psi $ such that 
\begin{equation*}
1 \geq F(z)F(z)^{\ast} \psi(F(z)F(z)^{\ast}) \geq |h(z)| \text{ for all } z \in \D
\end{equation*}
implies $ h \in \mathcal{I} ( \{ f_j\}_{j=1}^n) $?  Many authors, independently, have considered this question, including  Cegrell \cite{cegrell1}, Pau \cite{pau}, Trent \cite{estimate}, and Treil \cite{T3}. It is Treil who has given the best known sufficient condition for ideal membership.

We let $ H_{l^2}^{\infty}(\D) $ denote the Hilbert space of bounded, analytic functions that map $ \D $ to $ l^2 $. That is, a vector $ f $ in $ H_{l^2}^{\infty}(\D $ is an infinituple consisting of functions $ f_i \in H^{\infty}(\D) $ such that 
\begin{equation*}
\| f \|^2 = \sum_{i = 1}^{\infty} \sup_{z \in \D} |f_i(z)| < \infty \text{.}
\end{equation*}
We now give Treil's Theorem as follows:

\begin{tr}

Let $ F(z)= (f_1(z), f_2(z),...),  \; f_j  \in H^{\infty}(\mathbb{D})$, $F(z) F(z)^{\ast} \leq 1  \text{ for all } \; z\in \D$,  and $ h\in H^{\infty}(\mathbb{D})$ such that 
$$  F(z)F(z)^{\ast}\;  \psi \left( F(z)F(z)^{\ast}\right) \geq |h(z)|  \text{ for all }  z\in \mathbb{D},$$ where $\psi:[0,1]\rightarrow [0,1]$ is a non-decreasing  function such that $\int _{0}^{1} \frac{\psi(t)}{t} dt <\infty.$ Then there exists $G \in H_{l^2}^{\infty}(\D)$ such that $$F(z)G(z)^{T}= h(z),\;\; \text{for all} \;\; z\in \D \text{.}$$
\end{tr}

For an example of such a function $ \psi $ one can consider a function defined near 0 by
\begin{equation*}
\psi(t) = \frac{1}{(\ln t^{-2})(\ln_2 t^{-2}) \dots (\ln_n t^{-2})(\ln_{n+1} t^{-2})^{1+\epsilon}} \text{,}
\end{equation*}
  where $\ln_k(t)=\underbrace {\ln \ln ... \ln}_\text{ k times}(t)$  and $ \epsilon > 0.$ \\

For our paper, we consider three types of subalgebras of $H^{\infty}(\D)$. We use the fact that both the Corona Theorem and 
Wolff's Theorem hold on $H^{\infty}(\D)$ to find solutions contained within the given subalgebras. 

The first type is the
collection of subalgebras of the form $$ \C + BH^{\infty}(\D) = \lbrace \alpha + Bg : \alpha \in \C, g
\in H^{\infty}(\D) \rbrace, $$ where $ B $ is a fixed Blaschke product. This algebra was introduced and function problems were considered in J. Solasso \cite{sola}, M. Ragupathi \cite{pathi}, and Davidson, Paulsen, Ragupathi, and Singh \cite{DPRS}. In \cite{MSW}, Mortini, Sasane, and Wick proved the Corona Theorem for a finite number of generators, whereas the infinite version is due to Ryle and Trent \cite{{ryle},{ryle2}}. Just as in the $ H^{\infty} $ case, condition \eqref{eq:wolff} is not sufficient to guarantee ideal membership of the function $ h $ in the algebra $ \C + BH^{\infty}(\D) $, as can be shown by simple modification of  Rao's counterexample \cite{rao}.  Given $p < 2$ and a Blaschke product $B$ there exist functions $f, f_{1}, f_{2} \in BH^{\infty}$ such that they satisfy (2), but the equation $$ f_{1} g_{1} + f_{2}g_{2}=f $$ does not have a solution in $ BH^{\infty}$. 
  Note that in this example, $f_{1}(z)=f_2(z)=f(z)=0$ on the zero set $Z(B)$ of B. Since  the equation $f_{1} g_{1} + f_{2}g_{2}=f $ does not have a solution in $ BH^{\infty}$, obviously, it does not have a solution in $ \C + BH^{\infty}(\D) $.

\begin{thm} \label{CBH} 
Let $ F(z) = (f_1(z), f_2(z), \dots),  \; f_j \in \C + BH^{\infty}(\D),  F(z)F(z)^{\ast}\leq 1 \; \text{ for all } \; z\in \D $ and $
h \in \C + BH^{\infty}(\D) $, with $$ F(z)F(z)^{\ast}\;  \psi \left(F(z)F(z)^{\ast}\right)  \geq |h(z)| \; \text{ for all } \; z\in \D , $$ where $\psi$ is a function given as in Theorem B. Then there exists $ V = (v_1(z), v_2(z), \dots ), \; v_{j} \in \C + BH^{\infty}(\D)$ such that $$ F(z)V(z)^T = h(z) \; \text{ for all}\;  z \in \D.  $$

The solution vector $ V(z) $ is bounded as follows:
\begin{itemize}
\item[] $ \|V(z)\|_{l^2} \leq (1 + \frac{1}{\|F(\alpha)\|}_{l^2})C_0 $ if $ F(\alpha) \neq 0 $, and
\item[] $ \|V(z)\|_{l^2} \leq C_0 $ if $ F(\alpha) = 0 $,
\end{itemize} 
where $ \alpha $ is a zero of $ B(z) $ and  $ C_0 $ is the norm of the $ H^{\infty} $ solution obtained in \cite{T3}.
\end{thm}

\begin{rem} In Theorem \ref{CBH}, we use Treil's \cite{T3} $ H^{\infty} $ solution $V$ of $FV^T=h$ to find the solution of $FV^T=h$ in $\mathbb{C}+BH^{\infty}(\D)$. While it is  not stated directly in \cite{T3}, one can see it from the proof there that, in Treil's theorem, the solution can be estimated by a constant $C_0$ depending on the function $\psi$.
\end{rem}
\begin{cor} \label{CCBH} Let $ F= (f_1, f_2,  \dots) ,  \;  f_{j} \in \C + BH^{\infty}(\D)$ and $
h \in \C + BH^{\infty}(\D) $, with $$ 1 \geq [F(z)F(z)^{\ast}]^{\frac{1}{2}} \geq  |h(z)| \; \text{ for all } \; z\in \D .$$ Then there exists $ V= (v_1, v_2, \dots ), \; v_{j} \in \C + BH^{\infty}(\D)$ such that
\begin{equation*}
 F(z)V(z)^T = h^3(z) \;\;  \text{ for all } \; z\in \D  \text{.}
\end{equation*}

The solution vector $ V(z) $ is bounded by a constant $C_1$ as in Theorem \ref{CBH}, where $C_1$ is the estimate of the $H^{\infty}$
solution obtained in  \cite{estimate}.
\end{cor}

For the second type of subalgebra, let $ K \subset \Z_{+} $ and define $$ H_K^{\infty}(\D) = \lbrace f
\in H^{\infty}(\D) : f^{(j)}(0) = 0 \; \text{ for all } \;   j \in K \rbrace. $$ We consider those sets $ K $ for which $
H_K^{\infty}(\D) $ is an algebra under the usual product of functions. Obviously, not every set $ K
$ defines an algebra; for example, let $ K = \lbrace 2 \rbrace $. Though there is not a complete characterization of  the set $K$ for which $H_K^{\infty}(\D)$ is an algebra, Ryle and Trent \cite{ryle} have given certain criteria that the set $K$ must meet. In Lemma 2.1, we will state some of these criteria.
  For our purposes we assume $ K $
is finite. (We justify this assumption in the next section.)\\

We define algebras comprised of vectors with entries in $ H_K^{\infty}(\D) $ as follows:
\begin{align*} 
\mathcal{H}_{K, n}^{\infty}(\D) = \lbrace \lbrace f_j \rbrace_{j=1}^{n} : f_j
\in H_{K}^{\infty}(\D) \text{ for } j = 1, 2, \dots , n \\ 
\text{ and } \sup_{z \in \D} \sum_{j=1}^{n} \|f_j(z) \|^2 < \infty \rbrace. 
\end{align*} 
Multiplication here is entrywise, and
$n$ can be either a positive integer or $ \infty $. We write the elements of $\mathcal{H}_{K,
n}^{\infty}(\D)$ as row vectors, so that   $F\in \mathcal{H}_{K, n}^{\infty}(\D).$ 

Theorem 1.2 and Corollary 1.2 are analogues of Theorem 1.1 and Corollary 1.1, respectively, in this algebra. However, we need the additional assumption that $ F(0) \neq \mathbf{0} $.
\begin{thm} \label{HK1} 
Let $ F= (f_1, f_2, \dots)  \in \mathcal{H}_{K, n}^{\infty}(\D) $
and $ h \in H_K^{\infty}(\D) $, with $ 1 \geq F(z)F(z)^{\ast}\psi(F(z)F(z)^{\ast}) \geq |h(z)| \all z \in \D
$. Suppose also that $ F(0) \neq \mathbf{0} $. Then there exists $  V = (v_1, v_2, \dots )
\in \mathcal{H}_{K, n}^{\infty}(\D) $ such that 
\begin{equation*}
 F(z)V(z)^T = h(z) \all z \in \D
\end{equation*}
and
\begin{equation*}
\|V(z)\|_{l^2} \leq C_0 + \frac{\|G^{(k_p)}(0)\|_{l^2}}{{k_p}! \|F(0)\|_{l^2}} \text{.}
\end{equation*}
Here $ k_p $ is the largest element of $ K $, and $ G $ is an $ H^{\infty} $ solution obtained as in \cite{T3}.
\end{thm}

\begin{cor} 
Let $ F= (f_1, f_2, \dots) \dots) \in \mathcal{H}_{K, n}^{\infty}(\D) $
and $ h \in H_K^{\infty}(\D) $, with $ [F(z)F(z)^{\ast}]^{\frac{1}{2}} \geq |h(z)| \all z \in \D
$. Suppose also that $ F(0) \neq \mathbf{0} $. Then there exists $ V = (v_1, v_2, \dots )
\in \mathcal{H}_{K, n}^{\infty}(\D) $ such that 
\begin{equation*} 
F(z)V(z)^T = h^3(z) \all z \in \D \text{.}
\end{equation*} 
The solution vector $ V(z) $ is bounded as in Theorem \ref{HK1}, with $ C_0 $ replaced with $ C_1 $.
\end{cor}


For the third type of algebra, let $ K = \lbrace k_1, \dots , k_p \rbrace $ be a nontrivial finite subset of $
\Z_{+} $ such that $ H_K^{\infty}(\D) $ is an algebra, with $ k_1 < \dots < k_p $. For a fixed
Blaschke product, $ B $, we define 
\begin{equation*} 
H_{K(B)}^{\infty}(\D) = \lbrace
\sum_{\substack{j \notin K \\ j < k_p}} a_j B^j + B^{k_p + 1}g : g \in H^{\infty}(\D) \text{ and }
a_j \in \C \rbrace \text{,} 
\end{equation*} 
and we define $ \mathcal{H}_{K(B), n}^{\infty}(\D) $ similarly to $ \mathcal{H}_{K, n}^{\infty}(\D) $.
For $ F \in \mathcal{H}^{\infty}_{K(B),n}(\D) $, denote 
\begin{equation*}
F(z) = \underset{\substack{j \notin K \\ j < k_p}} \sum B^j(z)F_j + B^{k_p + 1}(z)F_{k_p+1}(z) \text{.}
\end{equation*}

\begin{thm} \label{HKB1} Let $ F= (f_1, f_2, \dots) \in \mathcal{H}_{K(B), n}^{\infty}(\D)
$ and $ h \in H_{K(B)}^{\infty}(\D) $, with $ 1 \geq F(z)F(z)^{\ast}\psi(F(z)F(z)^{\ast}) \geq |h(z)| \;\;  \text{ for all } \; z\in \D $. Suppose also that $ F_0 \neq \mathbf{0} $. Then there exists $ V = (v_1, v_2,
\dots ) \in \mathcal{H}_{K(B), n}^{\infty}(\D) $ such that 
\begin{equation*} 
F(z)V(z)^T = h(z) \;\;  \text{ for all } \; z\in \D 
\end{equation*}
and
\begin{equation*}
\|V(z)\|_{l^2} \leq \left(1 + \frac{1}{\|F_0\|}_{l^2} \right)C_0 \text{.}
\end{equation*}
\end{thm}

\begin{cor}
 Let $  F= (f_1, f_2, \dots) \in \mathcal{H}_{K(B), n}^{\infty}(\D)
$ and $ h  \in H_{K(B)}^{\infty}(\D) $, with $ [F(z)F(z)^{\ast}]^{\frac{1}{2}} \geq |h(z)| \all z
\in \D $. Suppose also that $ F_0 \neq \mathbf{0} $. Then there exists $ V = (v_1, v_2,
\dots )  \in \mathcal{H}_{K(B), n}^{\infty}(\D) $ such that 
\begin{equation*} 
F(z)V(z)^T = h^3(z) \all z \in \D \text{.} 
\end{equation*} 
The solution vector $ V(z) $ is bounded as in Theorem \ref{HKB1}, with $ C_0 $ replaced with $ C_1 $.
\end{cor}

Suppose we take the hypotheses of Theorem \ref{HK1}, but we allow $ F(0) = \mathbf{0} $. Since
$F(z)$ is a vector of holomorphic functions, we have $ F(z) = z^mF_m(z) $ for some $ m \in
\mathbb{N} $ where the entries of $ F_m(z) $ are holomorphic on $ \D $ and $ F_m(0) \neq \mathbf{0}$. One might attempt to continue in the vein of the proof to Theorem \ref{HK1} with $F_m $ in place of $ F $. Unfortunately, we need not expect $ F_m $ to lie in $
\mathcal{H}^{\infty}_{K, n}(\D) $ or any subalgebra thereof. We encounter a similar problem if we
allow $ F_0 = \mathbf{0} $ in \ref{HKB1} and factor $ B(z) $ off of $ F(z) $.

However, there are conditions under which  generalized ideal membership (and thus Wolff's Theorem)  still hold even if we allow the vectors
above to be zero. For a set $ K $ such that $ H^{\infty}_K(\D) $ is an algebra and $ m \in
\mathbb{N} $, $ m \notin K $, define $ K - m = \lbrace j - m : j \in K \text{ and } j > m \rbrace $.

\begin{thm} \label{HK2} 
Let $ F = (f_1, f_2, \dots)\in  \mathcal{H}_{K, n}^{\infty}(\D) $
and $ h \in H_K^{\infty}(\D) $, with $ 1 \geq F(z)F(z)^{\ast}\psi(F(z)F(z)^{\ast}) \geq |h(z)| \;\;  \text{ for all } \; z\in \D
$. Suppose also that $ F(z) = z^mF_m(z) $ with $ F_m(0) \neq \mathbf{0} $. If either 
\begin{itemize}
\item[(i)] $ K - m $ defines an algebra $ H^{\infty}_{K-m}(\D) $, or 
\item[(ii)] $ m > k_p $,
\end{itemize} 
then there exists $ V= (v_1, v_2, \dots ) \in \mathcal{H}_{K, n}^{\infty}(\D)
$ such that 
\begin{equation*} 
F(z)V(z)^T = h(z) \;\;  \text{ for all } \; z\in \D
\end{equation*} 
and
\begin{equation*}
\|V(z)\|_{l^2} \leq C_0 \text{.}
\end{equation*}
\end{thm}

\begin{thm} \label{HKB2} 
Let $ F = (f_1, f_2, \dots) \in \mathcal{H}_{K(B), n}^{\infty}(\D)
$ and $ h \in H_{K(B)}^{\infty}(\D) $, with $ 1 \geq F(z)F(z)^{\ast}\psi(F(z)F(z)^{\ast}) \geq |h(z)| \;\;  \text{ for all } \; z\in \D $. Suppose also that $ F_0 = \mathbf{0} $, and let $ j_1 > 0 $ be the greatest power of $ B $ common to all terms of $ F $. If either
\begin{itemize}
\item[(i)] $ K - j_1 $ defines an algebra $ H^{\infty}_{K-j_1}(\D) $, or
\item[(ii)] $ j_1 > k_p $,
\end{itemize}
then there exists $ V = (v_1, v_2, \dots ) \in \mathcal{H}_{K(B), n}^{\infty}(\D) $ such that 
\begin{equation*} 
F(z)V(z)^T = h(z) \;\;  \text{ for all } \; z\in \D
\end{equation*}  
and
\begin{equation*}
\|V(z)\|_{l^2} \leq C_0 \text{.}
\end{equation*}
\end{thm}

\section{Preliminaries}

Integral to the proofs of our theorems are ``Q-operators'' which are derived from the Kozsul complex
\cite{ryle}. As these operators have already been discussed in several papers, we will only give the
pertinent results here. Proofs of these results  and the identification of $Q^{(n)}$ operators and their matrices can be found in \cite{ryle}.

We let $ H \wedge K $ denote the exterior product between two Hilbert spaces $ H $ and $ K $, and $
l^2_{(n)} = \wedge_{i=1}^{n}l^2 $. In keeping with this notation, $ l^2_{(0)} = \C $.

Let $ \lbrace e_i \rbrace_{i = 1}^{\infty} $ denote the standard basis in $ l^2 $. If $ I_n $
denotes increasing $n$-tuples of positive integers and if $ (i_1, i_2, \dots , i_n) \in I_n $, we
let $ \pi_n = ( i_1, i_2, \dots , i_n ) $ and, abusing notation, we write $ \pi_n \in I_n
$. If we define $ e_{\pi_n} = e_{i_1} \wedge e_{i_2} \wedge \dots \wedge e_{i_n} $, then $ \lbrace
e_{\pi_n} \rbrace_{\pi_n \in I_n} $ is defined to be the standard basis for $ l^2_{(n)} $.

Let $ A = ( a_1 , a_2 , \dots ) \in l^2 $ and, for $ n = 1, 2, \dots $, define
\begin{equation*}
Q_A^{(n)\ast} : l^2_{(n)} \rightarrow l^2_{(n+1)}
\end{equation*}
by
\begin{equation*}
Q_A^{(n)\ast}(\underline{w}_n) = \overline{A} \wedge \underline{w}_n \text{,}
\end{equation*}
where $ \underline{w}_n \in l^2_{(n)} $.

We make note of some pertinent facts concerning these operators. First, we note that the entries of $ Q_A^{(n)} $ belong to the set $ \lbrace 0, \pm a_1 , \pm a_2 , \dots \rbrace $. Next, thanks to the anti-commutivity of the exterior algebra, $ \ran Q_A^{(n)\ast} \subset \ker Q_A^{(n+1)\ast} $, meaning
\begin{equation} \label{ker}
\ran Q_A^{(n+1)} \subset \ker Q_A^{(n)} \text{.}
\end{equation}
Finally, if $ B $ is also in $ l^2 $, then
\begin{equation} \label{relation}
(AB^T)I_{l^2} = B^T A + Q_A Q_B^T
\end{equation}
where $ I_{l^2} $ is an identity matrix and $ Q_{\Omega} = Q_{\Omega}^{(1)} $ for $ \Omega \in \lbrace A, B \rbrace $.

Note that $ Q_A^{(0)\ast} = \overline{A} $, so $ Q_A^{(0)} = A^T $. A construction of the matrix representation of $ Q_A^{(1)} $ can be found in \cite{ryle}.

We also draw on the following results from \cite{ryle}:

\begin{lemma}
Let $K\subseteq \mathbb{N}$ such that $ H_K^{\infty}(\D)$ is an algebra. Then 
\begin{align*}
& (i) \;\; {k_0}\in K \;\; \text{if and only if } \;\; \varphi (z)=z^{k_0}\in  H_K^{\infty}(\D).\\
& (ii)\text{ If} \;\; j ,k \notin K, \text{then}\;\;  j+k \notin K.\\
& (iii) \text{ Suppose} \;\; {k_0}\in K. \; \text{If}\;\;  1<j< k_0 \;\; \text{satisfies} \;\; j\notin K,\;\;  \text{then}\;\; k_0-j\in K.
\end{align*}

\end{lemma}

\begin{lemma} If $ H_K^{\infty}(\D) $ is an algebra, then there exists $ d \in \mathbb{N} $, a finite
set $ \lbrace n_i \rbrace_{i=1}^p \subset \mathbb{N} $ with $ n_1 < \dots < n_p $ and $ gcd(n_1,
\dots n_p) = 1 $, and a positive integer $ N_0 > n_p $ so that \begin{equation*} \mathbb{N} - K =
\lbrace n_1d, n_2d, \dots , n_pd, N_0d, (N_0 + j)d : j \in \mathbb{N} \rbrace \text{.}
\end{equation*} \end{lemma}

Lemma 2.2  tells us that the nontrivial sets $ K \subset \mathbb{N} $ for which $ H_K^{\infty}(\D)
$ is an algebra are the sets $ K $ for which there exist $ l_1 < \dots < l_r $ in $ \mathbb{N} $
with $ gcd(l_1, \dots , l_r) = d $ so that $ \mathbb{N} - K $ is the semigroup of $ \mathbb{N} $
generated by $ \lbrace l_1, \dots , l_r \rbrace $ under addition.

Thus the elements of $ H_K^{\infty}(\D) $ have the form \begin{equation*} F(z) = f_0 + f_1z^{n_1d} +
\dots + f_jz^{n_jd} + f_{j+1}z^{(n_j+1)d} + f_{j+2}z^{(n_j+2)d} + \dots \end{equation*} where $ f_i
\in \C $. Letting $ w = z^d $ yields \begin{equation*} F_1(w) = f_0 + f_1w^{n_1} + \dots +
f_{j-1}w^{n_{j-1}} + \sum_{k=0}^{\infty} f_{j+k}w^{n_j+k} \text{.} \end{equation*} Thus $ F_1(w) $
is contained in the algebra $ H_{K_1}^{\infty}(\D) $, where \begin{equation*} K_1 = \lbrace 1, \dots
, n_1-1, n_1+1, \dots , n_2-1, n_2+1, \dots , n_j-1 \rbrace \end{equation*} is a finite set.

The above argument suggests us that the problem of finding a solution to the ideal problem in $
H_K^{\infty}(\D) $, where $ K $ is infinite, can be reduced to two simpler steps. First, solve the
corresponding problem in $ H_{K_1}^{\infty}(\D) $,  where $ K_1 $ is finite as above. Then, take those
solutions in $ H_{K_1}^{\infty}(\D) $ and compose them with $ z^d $ in order to get the solution in
$ H_K^{\infty}(\D) $.

\section{The Proofs}

Our approach for each proof is similar. Since we are dealing with subspaces of $ H^{\infty}(\D) $, we use Treil's (Wolff's) Theorem to find a solution $ G \in H_{l^2}^{\infty}(\D) $ to
$ F(z)G(z)^T = h(z) $ ( respectively $ F(z)G(z)^T = h^3(z) $ ). Then we define $V$  as $ V(z)^T = G(z)^T + Q_{F(z)}X(z)^T, \; \text {for all } \;  z\in \D ,$  where $ Q_{F(z)}=Q^{(1)}_{ F(z)} $ and $ X \in H^{\infty}(\D)_{l^2}.$\\
Since  $Q^{(0)}_{ F(z)}=F(z), $ the inclusion (3) with $n=0$ implies that 
$$ F(z)Q_{F(z)}=Q_{F(z)}^{(0)} Q_{F(z)}^{(1)}=0  \; \text{for all} \; z\in \D.$$ 
So $F(z) V(z)^T =F(z) V(z)^T=h(z)$ (respectively $F(z) V(z)^T=h^3(z))$. That means $V$ is a solution. Therefore, our goal is to show that V belongs to the appropriate subalgebra. 

\begin{proof}[Proof of Theorem \ref{CBH}]

Let $ F \in (\C + BH^{\infty}(\D))_{l^2} $, $ h \in \C + BH^{\infty}(\D) $, and suppose
$$ 1 \geq F(z)F(z)^{\ast}\;  \psi \left(F(z)F(z)^{\ast}\right)  \geq |h(z)| \;\;  \text{ for all } \; z\in \D. $$
 By Treil's theorem, there exists $ G \in H_{l^2}^{\infty}(\D) $ such that
\begin{equation*} 
F(z)G(z)^T = h(z) \;\;  \text{ for all } \; z\in \D \text{.} 
\end{equation*}

Write $ F(z) = F_c + B(z)F_B(z) $, where $ F_c \in l^2 $ and $ F_B \in H_{l^2}^{\infty}(\D) $.
Also, write $ h(z) = h_c + B(z)h_B(z) $, with $ h_c \in \C $ and $ h_B \in H^{\infty}(\D) $. We consider two cases.

Suppose first that $ F_c \neq \mathbf{0} $.\\
 By \eqref{relation}, we have 
\begin{align*}
&h(z)I = (F(z)G(z)^T)I = G(z)^TF(z) + Q_{F(z)}Q_{G(z)}^T \\
\implies & (h_c + B(z)h_B(z))I= G(z)^T(F_c + B(z)F_B(z)) + Q_{F(z)}Q_{G(z)}^T. 
\end{align*} 

 Thus 
\begin{align*} 
&(h_c + B(z)h_B(z))F_c^{\ast} = G(z)^T(F_c + B(z)F_B(z))F_c^{\ast} + Q_{F(z)}Q_{G(z)}^TF_c^{\ast} \\ 
\implies &   h_cF_c^{\ast} + B(z)(h_B(z) - G(z)^TF_B(z))F_c^{\ast} = G(z)^TF_cF_c^{\ast} +
Q_{F(z)}Q_{G(z)}^TF_c^{\ast} \\ 
\implies &  \frac{h_c}{\|F_c\|^2}F_c^{\ast} + B(z)\left(  h_B(z) -
G(z)^TF_B(z)\right)\frac{F_c^{\ast}}{\|F_c\|^2} = G(z)^T + Q_{F(z)}Q_{G(z)}^T\frac{F_c^{\ast}}{\|F_c\|^2} \text{.}
\end{align*} 
The right hand side of the last equation is clearly a solution $ V(z)^T $ to $
F(z)V(z)^T = h(z) $, while the left hand side shows this solution is in $ (\C + BH^{\infty}(\D))_{l^2} $. Thus we take $ X(z)^T = Q_{G{(z)}}^T\frac{F_c^{\ast}}{\|F_c\|^2} $. 

For the norm estimate, we have $$ \|V\|_{\infty} \leq
\left(1 + \frac{1}{\|F_c\|^2}\right)\|G\|_{\infty} .$$

Now suppose  $ F_c = \mathbf{0} $. We thus have $$ [\vert B(z)\vert^2 F_B(z)F_B(z)^{\ast}]\;  \psi \left( \vert B(z)\vert ^{2} F_B(z)F_B(z)^{\ast}\right) \geq |h_c+ B(z)h_B(z)|  \;\;  \text{ for all } \; z\in \D. $$ Letting $ z = \alpha $, where $ \alpha $ is a zero of $ B(z) $, we see that $ h_c = 0 $. Thus 
$$
[\vert B(z)\vert^2 F_B(z)F_B(z)^{\ast}]\;  \psi \left( \vert B(z)\vert ^{2} F_B(z)F_B(z)^{\ast}\right) \geq |B(z)||h_B(z)| $$
This implies that we can factor at least one more $B$ out from $h_B$. Since $ \psi $ is increasing on $ [0, 1] $ and $ |B(z)| \leq 1 $ on $ \D $, we get  
$$[F_B(z)F_B(z)^{\ast}]\;  \psi \left(F_B(z)F_B(z)^{\ast}\right) \geq |h_{B_1}(z)|,$$ where $h_B=Bh_{B_1}$. We should note that $h_{B_1}$ may contain more $B$'s.

By Treil's Theorem, there exists $ G_1 \in H_{l^2}^{\infty}(\D) $ such that
\begin{align*} 
& F_B(z)G_1(z)^T = h_{B_1}(z) \\ 
\implies &  F(z)B(z)G_1(z)^T = B^2(z)h_{B_1}(z) = h(z) \;\;  \text{ for all } \; z\in \D \text{.} 
\end{align*} 
Thus, $ B(z)G_1(z)^T $ is the solution we seek.

We also see that $ \|B(z)G_1(z)^T\|_{\infty} \leq \|G_1(z)\|_{\infty}  $. This completes the proof.

\end{proof}

\begin{proof}[Proof of Corollary \ref{CCBH}]
The proof of this Corollary is similar to the proof of Theorem \ref{CBH}. We replace $h$ with $h^3$ and use Wolff's Theorem in $H^{\infty}(\D)$.
As in Theorem \ref{CBH}, we get a solution $V$ of $FV^{T}=h^{3}$ in $ (\C + BH^{\infty}(\D))_{l^2} $ satisfying  $$ \|V\|_{\infty} \leq
\left(1 + \frac{1}{\|F_c\|^2}\right)\|G_1\|_{\infty},$$ where $G_1$ is the $\left(H^{\infty}(\D)\right)_{l^2}$ solution of $F(z)G(z)^{T}=h^{3}(z).$ For the norm estimate, we draw upon the estimate in \cite{estimate} with $ \psi(t) =
t^{\frac{1}{2}} $.
\end{proof}

The proofs of Theorems \ref{HK1} through \ref{HKB2} are by induction. Since in each case a solution in $ H^{\infty}(\D) $ exists, by Treil, our use of induction is justified. Similarly,  Corollaries can be obtained using Wolff's Theorem instead of Treil's, as in Corollary 1.1. 

We denote $ K_{p-1} = K - \{{k_p}\} $, where $ k_p $ is the largest member of $ K $. If $ H^{\infty}_K(\D) $ is an algebra, then so is $ H^{\infty}_{K_{p-1}}(\D) $.

\begin{proof}[Proof of Theorem \ref{HK1}]

Let $ F  \in \mathcal{H}_{K, n}^{\infty}(\D) $, $ h  \in H_K^{\infty}(\D) $ such that $
F(0) \neq \mathbf{0} $ and 
\begin{equation*} 
1 \geq F(z)F(z)^{\ast}\psi(F(z)F(z)^{\ast}) \geq |h(z)| \;\;  \text{ for all } \; z\in \D \text{.}\ 
\end{equation*} 
By induction, there exists $ G \in \mathcal{H}_{K_{p-1},
n}^{\infty}(\D) $ with 
\begin{equation*}
 F(z)G(z)^T = h(z) \;\;  \text{ for all } \; z\in \D \text{.}
\end{equation*}

We denote ``$k_p$'' by ``$k$'', and we let \begin{equation*} X(z)^T =
\frac{1}{k!}\frac{Q_{F(0)}^{\ast}G^{(k)}(0)^T}{F(0)F(0)^{\ast}}z^k \in H_{K_{p-1}}^{\infty}(\D)
\text{.} \end{equation*} We consider \begin{equation*} V(z)^T = G(z)^T - Q_{F(z)}X(z)^T \text{.}
\end{equation*} We see that $$ F(z)V(z)^T = h(z) \;\;  \text{ for all } \; z\in \D \;\;  \text{and} \;\;  V(z) \in
\mathcal{H}_{K_{p-1}, n}^{\infty}(\D). $$ We must show that $ V^{(k)}(0) = \mathbf{0} $. But by
\eqref{relation}, \begin{align*} V^{(k)}(0)^T &= G^{(k)}(0)^T - \sum_{j=0}^k \left(
\begin{array}{c}k\\j\end{array}\right) Q_{F(0)}^{(k-j)}X^{(j)}(0)^T \\ &= G^{(k)}(0)^T -
Q_{F(0)}X^{(k)}(0)^T \\ &= G^{(k)}(0)^T -
\frac{Q_{F(0)}Q_{F(0)}^{\ast}}{F(0)F(0)^{\ast}}G^{(k)}(0)^T \\ &=
\frac{F(0)^{\ast}F(0)}{F(0)F(0)^{\ast}}G^{(k)}(0)^T \text{.} \end{align*}

Our proof thus depends on establishing that $ F(0)G^{(k)}(0)^T = 0 $. But $ F(z)G(z)^T = h(z) $ on
$ \D $, and $ h \in H_K^{\infty}(\D) $. Differentiating $ k $ times and evaluating at 0, we
obtain \begin{equation}\label{woop} \sum_{j=1}^k \left( \begin{array}{c} k\\j \end{array} \right)
F^{(k-j)}(0)G^{(j)}(0)^T = 0 \text{.} \end{equation} Since $ G(z) \in
\mathcal{H}_{K_{p-1},n}^{\infty}(\D) $, we have $ G^{(j)}(0) = \mathbf{0} \; \text{ for all} \;  j \in K_{p-1} $.
If $ j \notin K $ and $ j < k $, we have $ k - j \in K $, so $ F^{(k-j)}(0) = \mathbf{0} $. Thus
\eqref{woop} becomes \begin{equation*} F(0)G^{(k)}(0)^T = 0 \end{equation*} which is the desired
result.

For the norm estimate, we observe that 
\begin{align*}
 \|V\|_{\infty} & \leq \|G\|_{\infty} + \frac{1}{k!}\left \|Q_{F(z)} \frac{Q_{F(0)}^{\ast}G^{(k)}(0)^T}{F(0)F(0)^{\ast}}z^k \right\|_{\infty} \\
& \leq \|G\|_{\infty} +  \frac{1}{k!} \|Q_{F(z)}\|_{\infty} \frac{\|Q_{F(0)}^{\ast}\|_{l^2}\|G^{(k)}(0)\|_{l^2}}{\|F(0)\|_{l^2}^2} \\
& \leq \|G\|_{\infty} + \frac{1}{k!} \frac{\|G^{(k)}(0)\|_{l^2}}{\|F(0)\|_{l^2}} \text{.}
\end{align*}

\end{proof}

\begin{proof}[Proof of Theorem \ref{HK2}]

Observe first that $ m \notin K $, or else we would have $
\mathbf{0} = \frac{d^m}{dz^m}[z^mF_m(z)]|_{z=0} = m!F_m(0) $, and $ F_m(0) \neq \mathbf{0} $, by assumption. 
Similarly to Theorem \ref{CBH} in the case where $ F_c = \mathbf{0} $, we have, by Treil, $ F_m(z)G_m(z)^T = h_m(z) \all z \in \D $, where $ G_m \in H^{\infty}(\D) $
and $ h(z) = z^{2m}h_m(z) $. Thus
\begin{equation*}
F(z)z^mG_m(z)^T = z^{2m}h_m(z) = h(z) \all z \in \D \text{.}
\end{equation*}

We wish to show $ z^{m}G_m(z) \in \mathcal{H}_{K, n}^{\infty}(\D) $. If $ m > k_p $, the result is immediate.

If $ m < k_p $, then suppose  that $ K - m $ defines an algebra $ H^{\infty}_{K-m}(\D) $. Since $ m \notin K $, $ K - m \subset K_{p-1} $ and thus $ H_{K_{p-1}}^{\infty}(\D) \subset H_{K-m}^{\infty}(\D) $. Using induction, we can thus take $ G \in H^{\infty}_{K-m}(\D) $. Now, $
z^mG_m(z) \in \mathcal{H}_{K,n}^{\infty}(\D) $, for if $ j \in K $, then
\begin{equation*} 
\frac{d^j}{dz^j}(z^mG_m(z))|_{z=0} = \left(\begin{array}{c}
j\\m\end{array}\right)m!G_m^{(j-m)}(0) = \mathbf{0} \text{.} 
\end{equation*} 
(We assume here that $ j > m $. If $ j < m $, the result is trivial.) 

Finally, $ \|V\|_{\infty} \leq \|G_m\|_{\infty} $.

\end{proof}

For $ F \in H^{\infty}_{K(B)}(\D) $, denote $ F(z) = B^{j_1}(z)F_{j_1} + \dots + B^{j_n}(z)F_{j_n} + B^{k_p+1}(z)F_{k_p+1}(z)  $, where $ F_{j_i} \in \C $ for $ i = 1, \dots , n $ and $ F_{k_p+1}(z) \in \left(H^{\infty}(\D)\right)^n $. We inductively assume there exists a solution $ G(z) = B^{j_1}(z)G_{j_1} + \dots + B^{j_n-1}(z)G_{j_n-1} + B^{j_n}(z)G_{j_n}(z) $. One can check to see that $ G $ belongs to a subalgebra containing $ H^{\infty}_{K(B)}(\D) $.

\begin{proof}[Proof of Theorem \ref{HKB1}]

Since $ F_0 \neq \mathbf{0} $, denote $ j_1 = 0 $. Proceeding as in the proof of Theorem \ref{CBH} in the case $ F_c \neq \mathbf{0} $, we obtain 
\begin{align*} 
& [h_{0} + B^{j_2}(z)h_{j_2} + \dots + B^{j_n}(z)h_{j_n} + B^{k_p+1}(z)h_{k_p+1}(z)] \frac{F_0^{\ast}}{\|F_0\|^2} \\ 
&- G(z)^T[B^{j_2}(z)F_{j_2} + \dots + B^{j_n}(z)F_{j_n} + B^{k_p+1}(z)F_{k_p+1}(z)] \frac{F_0^{\ast}}{\|F_0\|^2} \\ 
&= G(z)^T + Q_{F(z)}Q_{G(z)}^T\frac{F_0^{\ast}}{\|F_0\|^2} \text{.} 
\end{align*} 
The remainder of the proof consists of showing that the left-hand side of this equation lies in $ \mathcal{H}^{\infty}_{K(B),n}(\D) $. We observe that
\begin{align*}
& G(z)^T[B^{j_2}(z)F_{j_2} + \dots + B^{j_n}(z)F_{j_n} + B^{k_p+1}(z)F_{k_p+1}(z)] \\
&= [G_0 + B^{j_2}(z)G_{j_2} + \dots + B^{j_n-1}(z)G_{j_n-1}]^T \\
& \times  [B^{j_2}(z)F_{j_2} + \dots + B^{j_n}(z)F_{j_n} + B^{k_p+1}(z)F_{k_p+1}(z)] \\
&+ B^{j_n}(z)G_{j_n}(z)^T [B^{j_2}(z)F_{j_2} + \dots + B^{j_n}(z)F_{j_n} + B^{k_p+1}(z)F_{k_p+1}(z)]
\end{align*}

The first term is clearly in $ B(\C^n, \mathcal{H}^{\infty}_{K(B),n}(\D)) $. Since, for $ i = 2, \dots , n $, $ j_n + j_i \notin K $, we must have $ j_n + j_1 > k_p $. This shows that the second term is also in $ B(\C^n, \mathcal{H}^{\infty}_{K,n}(\D)) $. Thus $ G(z)^T + Q_{F(z)}Q_{G(z)}^T\frac{F_0^{\ast}}{\|F_0\|^2} \in \mathcal{H}^{\infty}_{K(B),n}(\D) $.

The norm estimate is obtained as in the proof of Theorem \ref{CBH}.

\end{proof}

\begin{proof}[Proof of Theorem \ref{HKB2}]

Since $ F_0 = \mathbf{0} $, we have $ F(z) = B^{j_1}(z)F_{\alpha}(z) $, where $F_{\alpha} \in \left(H^{\infty}(\D)\right)^n $ and $ F_{\alpha}(z) $ has a nonzero constant term. As in the proof of Theorem \ref{CBH} in the case where $ F_c = \mathbf{0} $, there exists $ G_{\alpha} \in H^{\infty}(\D) $ such that $ F_{\alpha}(z)G_{\alpha}(z)^T = h_{\alpha}(z) \all z \in \D $, where $ B^{2j_1}(z)h_{\alpha} = h(z) $. Thus
\begin{equation*}
F(z)B^{j_1}(z)G_{\alpha}(z)^T = B^{2j_1}(z)h_{\alpha}(z) = h(z) \all z \in \D \text{.}
\end{equation*}

If $ j_1 > k_p $, then $ B^{j_1}(z)G_{\alpha}(z) = B^{k_p+1}(z)[B^{j_1-k_p-1}(z)G_{\alpha}(z)] \in \mathcal{H}^{\infty}_{K(B)}(\D) $, and we are done.

If $ j_1 < k_p $, then suppose $ K - j_1 $ defines an algebra $ H^{\infty}_{K-j_1}(\D) $. Then $ F_{\alpha} \in \mathcal{H}^{\infty}_{(K-j_1)(B)}(\D) $ and since the constant term of $ F_{\alpha} $ is nonzero, then by Theorem \ref{HKB1} we may assume $ G_{\alpha} \in \mathcal{H}^{\infty}_{(K-j_1)(B)}(\D) $. Then $ B^{j_1}(z)G_{\alpha}(z) \in \mathcal{H}^{\infty}_{K(B)}(\D) $.

Finally, $ \|V\|_{\infty} \leq \|G_{\alpha}\|_{\infty} $.

\end{proof}

\section{Further Results and Questions}

\subsection{Radical Ideals}

Consider the radical of the ideal generated by the functions $ f_i $,  
\begin{equation*}
\rad(\lbrace f_i \rbrace_{i=1}^n ) = \lbrace h \in H^{\infty} : \exist q \in \mathbb{N} \text{ with } h^q \in \mathcal{I}(\lbrace f_j \rbrace_{j=1}^n) \rbrace \text{.}
\end{equation*}
Equation \eqref{eq:wolff} actually provides a characterization for membership in the radical ideal. We obtain similar characterizations for algebras of form $ \C + BH^{\infty}(\D) $ and $ H^{\infty}_K(\D) $. For the first type of algebra this result is immediate, but the second type of algebra requires a bit of discussion.

Let $ F \in \mathcal{H}^{\infty}_{K,n}(\D) $ and $ h \in H^{\infty}_K(\D) $. Clearly, if $ h \in \rad(\mathcal{I}) $, where $ \mathcal{I} $ is the ideal generated by the entries in $ F(z) $, then 
\begin{equation*}
M[F(z)F^{\ast}(z)]^{\frac{1}{2}} \geq | h^q(z)| \all z \in \D
\end{equation*}
for some $ M > 0 $, $ q \in \mathbb{N} $. For the converse, we have two cases. If $ F(0) \neq \mathbf{0} $, the result follows from Theorem \ref{HK1}. If $ F(0) = \mathbf{0} $, then $ F(z) = z^mF_m(z) $ and $ h(z) = z^mh_m(z) $ as in the proof of Theorem \ref{HK2}. We use Wolff's Theorem to obtain a $ G \in \left(H^{\infty}(\D)\right)^n $ and $ q \in \mathbb{N} $ such that $ F(z)G(z)^T = h^q(z) \all z \in \D $. Take $ L \in \mathbb{N} $ such that $ mL > k_p $. Then
\begin{align*}
h^{p+L}(z) &= F(z)[h^L(z)G(z)^T] = F(z)[z^{mL}h_m(z)G(z)^T] \text{.}
\end{align*}
We thus take $ U(z)^T = z^{mL}h_m(z)G(z)^T $. Since $ mL > k_p $, $ U(z) \in H^{\infty}_K(\D) $. This shows that $ h \in \rad(\mathcal{I}) $.

\subsection{A Full Extension of Wolff's Theorem}

The added assumption in Theorem \ref{HK2} that $ H_{K-m}^{\infty}(\D) $ is an algebra (as well as the similar assumption in Theorem \ref{HKB2}) was necessary for our proof. For example, if $ K = \{1, 2, 5\} $, then $ H_K^{\infty}(\D) $ is an algebra, but $ K - 3 = \lbrace 2 \rbrace $ does not define an algebra. It remains an open question whether the generalized ideal result and Wolff's Theorem can be fully extended to the subalgebras $ H_K^{\infty}(\D) $ and $ H_{K(B)}^{\infty}(\D) $.\\

\subsection{Improving Estimates for $F(\alpha)$ Near Zero}

In Theorem \ref{CBH}, the norm estimate for the solution improves as $ F(\alpha) \rightarrow \infty $, and explodes as $ F(\alpha) \rightarrow 0 $; however, if 0 is actually attained, the best estimate results. We encounter a similar dilemma with the other theorems. In each case, this is due to the construction of the solution, but it is rather counterintuitive. One would like to find a solution that exhibits better behavior as $ F(\alpha) $ heads to zero.\\

Acknowledgement: The authors would like to thank the referee for the careful review and for the help in improving the presentation of the paper.

\end{document}